\documentclass[english]{article}
\usepackage[T1]{fontenc}
\usepackage[latin9]{inputenc}
\usepackage{geometry}
\usepackage{color}
\usepackage{units}
\usepackage{mathtools}
\usepackage{amsmath}
\usepackage{amsthm}
\usepackage{amssymb}
\usepackage{url}
\usepackage{array}
\usepackage{units}
\usepackage{multirow}
\usepackage{fixltx2e}

\makeatletter
\theoremstyle{plain}
\newtheorem{thm}{\protect\theoremname}
  \theoremstyle{remark}
  
  \theoremstyle{plain}
  \newtheorem{lem}{\protect\lemmaname}
  \theoremstyle{plain}

\usepackage{dsfont}
\usepackage{mathtools}

\makeatother

\usepackage{babel}
  \providecommand{\lemmaname}{Lemma}
  \providecommand{\propositionname}{Proposition}
  \providecommand{\remarkname}{Remark}
  \providecommand{\theoremname}{Theorem}
  \providecommand{\tabularnewline}{\\}

\begin{document}

\date{}

\title{There is no Khintchine threshold for metric pair correlations}

\author{Christoph Aistleitner, Thomas Lachmann, and Niclas Technau}
\maketitle
\begin{abstract}
We consider sequences of the form $\left(a_{n} \alpha\right)_{n}$ mod 1, where $\alpha\in\left[0,1\right]$
and where $\left(a_{n}\right)_{n}$ is a strictly increasing sequence of positive
integers. If the asymptotic distribution of the pair correlations 
of this sequence follows the Poissonian model for almost all $\alpha$ 
in the sense of Lebesgue measure, we say that $(a_n)_n$ 
has the metric pair correlation property. Recent research 
has revealed a connection between the metric theory of pair correlations 
of such sequences, and the additive energy 
of truncations of $(a_n)_{n}$. Bloom, Chow, Gafni and Walker 
speculated that there might be a convergence/divergence criterion 
which fully characterises the metric pair correlation property 
in terms of the additive energy, similar to Khintchine's criterion 
in the metric theory of Diophantine approximation. 
In the present paper we give a negative answer to such speculations, 
by showing that such a criterion does not exist. 
To this end, we construct a sequence $(a_n)_n$ 
having large additive energy which, however, 
maintains the metric pair correlation property.
\end{abstract}

\section{Introduction}

\global\long\def\eps{\varepsilon}

\global\long\def\deelta{\nicefrac{1}{4} + \nicefrac{\varepsilon}{3}}

Let $x_1, \dots, x_N$ be numbers in the unit interval. 
The distribution of the pair correlations of these numbers 
is described by the function 
\begin{equation}\label{eq: definition pair correlation function}
R(s,N) = \frac{1}{N} 
\left\{ 1 \leq i \neq j \leq N:~\|x_i - x_j\| 
\leq s/N \right\}, \qquad s \geq 0,
\end{equation}
where $\| \cdot \|$ denotes the distance to the nearest integer. 
If for an infinite sequence $(x_n)_n$ 
we have 
\begin{equation}\label{eq: counting functing tends to 2s}
R(s,N) \to 2s
\end{equation} 
for all $s \geq 0$, then we say that 
the distribution of pair correlations 
is (asymptotically) \emph{Poissonian}. 
Note that a sequence of independent, 
identically distributed (i.i.d.) random points, 
picked from a uniform distribution on $[0,1]$, 
almost surely has Poissonian pair correlations. 
The term ``Poissonian'' comes from a similarity 
with the distribution of the spacings between 
points in a Poisson process, which, however, 
only becomes really meaningful 
when also considering higher correlations (triple, 
quadruple etc.) or so-called level	 spacings 
(which are in general much more difficult to handle than pair correlations).\\

The interest in such problems goes back to a paper of Berry and Tabor \cite{BT}, 
where they gave a conjectural framework for the distribution of energy spectra 
of integrable quantum systems (see \cite{mark} for a survey). Their model led to strong mathematical interest 
in distributional properties of spacing of sequences such as $(n \alpha)_n$ mod 1 
(corresponding to the ``harmonic oscillator'') and $(n^2 \alpha)_n$ mod 1 
(corresponding to the ``boxed oscillator''). 
The case of $(n \alpha)_n$ is easier to analyse; 
one can use considerations based on continued fractions 
to show that the pair correlations of this sequence 
cannot be Poissonian for any $\alpha$, since for some $N$ 
the initial segment $(\alpha, 2\alpha, \dots, N \alpha)$ mod 1 is too regularly spaced. 
The case of $(n^2 \alpha)_n$ is much harder and is far from being well-understood. 
It is conjectured that the pair correlations for this sequence should be Poissonian, 
unless $\alpha$ is very well approximable by rationals; however, 
there exist only some partial results in this direction 
(see for example \cite{hb,rsz,tru}). 
From the metric perspective, the situation is easier: 
it is known that the pair correlations of $(n^2 \alpha)_n$ mod 1 
are Poissonian for almost all $\alpha$, in the sense of Lebesgue measure. 
The same is true if $(n^2)_n$ is replaced by $(n^d)_n$ 
for some integer $d \geq 3$, 
or by an exponentially growing sequence $(a_n)_n$ of integers, 
see \cite{ruds,rz}. We denote this property by saying that these sequences 
have the \emph{metric pair correlation property}. 
In a recent paper \cite{all}, a connection was established between 
the question whether a sequence has the metric pair correlation property, 
and the asymptotic order of its so-called additive energy. 
Let $(a_n)_n$ be a sequence of distinct positive integers, 
let $A_N$ denote its initial segment $a_1, \dots, a_N$, and denote by
$E(A_N)$ the {\em additive energy} of $A_N$, which is defined as
\begin{equation}\label{eq: definition additive energy}
E(A_N)= \# \{n_1, n_2, n_3, n_4 \leq N:~a_{n_1} + a_{n_2} 
= a_{n_3} + a_{n_4} \}.
\end{equation}
Trivially, the additive energy is always between $N^2$ and $N^3$. Throughout this paper we will use the formulation ``the order of the additive energy of a sequence'', when more precisely speaking we mean the order (as a function of $N$) of the additive energy of the $N$ first elements of the sequence.\\

The main results of \cite{all} say that a sequence
has the metric pair correlation property if its additive energy is of order at most $N^{3-\varepsilon}$ for some $\varepsilon > 0$, 
while it does not have the metric pair correlation property if the additive energy exceeds
$c N^{3}$ for infinitely many $N$ for some constant $c>0$. 
This fits together very well with the examples from above, 
since sequences of the form $(n^d)_n$ for $d \geq 2$ 
and lacunary sequences are known to have very small additive energy, 
while the additive energy of the sequence $a_n = n, ~n \geq 1,$ 
is of the maximal possible order.\\

So the general philosophy is that a sequence has 
the metric pair correlation property as soon as its additive energy is 
slightly below the maximal possible order. 
However, a precise threshold is not known. 
Some results in this direction are:
\begin{itemize}
 \item The primes do not have the metric pair correlation property, 
 as shown by Walker \cite{walk}. 
The additive energy of the sequence of primes is roughly of order $\frac{N^3}{\log N}$.
 \item There exists a sequence having additive energy of order 
$\frac{N^3}{\log N \log \log N}$ which does not have 
the metric pair correlation property \cite{lach}.
 \item For every $\varepsilon>0$ there exists a sequence having additive energy 
of order $\frac{N^3}{\log N (\log \log N)^{1+\varepsilon}}$ 
which has the metric pair correlation property (unpublished, 
but not difficult to construct using methods from \cite{bcgw,lach}).
 \end{itemize}

These results indicate that there is a sort of transitional behaviour 
when the additive energy lies around the ``critical'' order of roughly 
$\frac{N^3}{\log N \log \log N}$. The methods used in \cite{bcgw,lach} 
indicate a close connection between this sort of question and problems 
from metric Diophantine approximation, where the classical theorem of Khintchine
gives a zero-one law in terms of the convergence/divergence of the series of measures of the target intervals 
(see for example \cite{harman}). 
It is tempting to speculate that a similar convergence/divergence criterion 
might also exist for the metric theory of pair correlations, 
where the crucial quantity is the additive energy of 
$(a_n)_n$. 
This idea was discussed in a recent paper of Bloom, Chow, Gafni, and Walker \cite{bcgw}, 
where they noted that there ``appears to be reasonable 
evidence to speculate a sharp Khintchine-type threshold, 
that is, to speculate that the metric Poissonian property 
should be completely determined by whether or not a certain sum 
of additive energies is convergent or divergent''. 
They raise the following problem, which they call the ``Fundamental Question'':

\begin{quote}Is it true that if $E\left(A_{N}\right)\sim N^{3}\psi\left(N\right)$,
for some weakly decreasing function $\psi:\mathbb{Z}_{\geq1}\rightarrow\left[0,1\right]$,
then $\left(a_{n}\right)_{n}$ is metric Poissonian if and only if
\begin{equation} 
\sum_{N\geq1}\psi\left(N\right)/N\label{eq: sum over normalized additive energies}
\end{equation}
converges?
\end{quote}

In the present paper, we show 
that the answer to the question above is negative,
and that the metric pair correlation property 
cannot be fully characterised in terms of the additive energy alone. 
For this purpose, we construct a sequence $(a_n)_n$ 
whose additive energy is of order roughly $N^3/(\log N)^{3/4}$,
and which \emph{does} have the metric pair correlation property. 
More precisely, we prove the following theorem.
\begin{thm}
\label{thm: main theorem}For every $\varepsilon>0$ 
there exists a strictly increasing sequence $\left(a_{n}\right)_{n}$ 
of positive integers which has the metric pair correlation property, 
and whose additive energy satisfies
\begin{equation}
E\left(A_{N}\right) \gg \frac{N^{3}}{(\log N)^{\nicefrac{3}{4}+\varepsilon}}. \label{eq: bound for the additive energy of truncations of the final sequences}
\end{equation}
\end{thm}
Note that the additive energy of the sequence from the conclusion of Theorem 1 is significantly larger 
than the putative threshold, which is rather around $N^3/\log N$. Furthermore, as the examples above showed, the additive energy of a sequence which does \emph{not} have the metric pair correlation property can be of asymptotic order $N^3$, but it can also be of asymptotic order roughly $N^3/\log N$. Thus the metric theory of pair correlations 
cannot simply be reduced to a convergence/divergence criterion 
in terms of the additive energy alone. Instead, the picture 
is more complicated and looks as follows:
\begin{itemize}
 \item If the additive energy is below a certain threshold, 
then the sequence does have the metric pair correlation property. 
 \item If the additive energy is above a certain threshold 
(for infinitely many $N$), then the sequence cannot have the metric pair correlation property. (This threshold is different from the one in the point above.)
 \item Between these upper and lower thresholds there is a transition zone, 
where knowing the additive energy alone is not sufficient 
to determine the metric pair correlation behaviour of the sequence. 
Thus, in this range the metric pair correlation property is determined 
by some additional arithmetic properties of the sequence.
\end{itemize}

We note that while our result 
says that the metric pair correlation property cannot be characterised 
in terms of the additive energy alone, it leaves the problem of finding 
some other way of characterising the metric pair correlation property 
in terms of some arithmetic properties of $(a_n)_n$. It is likely that there 
is a zero-one law in the metric theory of pair correlations, 
but actually even this is not known. Also, our result still leaves 
questions concerning the quantitative connection between additive energy 
and the metric theory of pair correlations. For example, is it possible that 
a sequence having additive energy of order $N^3/(\log \log N)$ 
also has the metric pair correlation property?
In the other direction, is it possible that the additive energy is of order $N^3/(\log N)^2$ 
and the sequence does not have the metric pair correlation property?\footnote{While the present paper was being refereed, a paper of Bloom and Walker addressing this question appeared on the arXiv. They proved that there exists an (unspecified) constant $C>1$ such that a sequence has the metric pair correlations property whenever its additive energy is of asymptotic order at most $N^3/(\log N)^C$, see \cite{bw}.}
Closing the gaps in our knowledge in this field would be very desirable, 
as phenomena from both additive combinatorics 
and Diophantine approximation seem to be at work here.

\section{Preliminaries}

\subsection{Construction of the sequence}

We will construct our sequence $(a_n)_n$ as the concatenation 
of successive ``blocks''. All these blocks are either 
finite arithmetic or finite geometric progressions. 
The geometric blocks will contain the majority 
of the numbers which constitute the final sequence, 
but they will not be responsible for making 
the additive energy of the final sequence large, 
since geometric progressions always have small additive energy. 
The contribution of the geometric blocks 
to the distribution of pair correlations will be ``random'', in accordance with the well-known heuristics that lacunary systems exhibit properties which are also shown by independent random systems (see for example \cite{rz} for this phenomenon in connection with pair correlations, and \cite{a_survey} for the wider context). In our context ``random'' behaviour means Poissonian behaviour of the pair correlations, so the geometric blocks are responsible that the final sequence is metric Poissonian. The arithmetic blocks contain only a minority of all the elements of the final sequence, while being responsible for making the additive energy large. The main task will be to show that while these arithmetic blocks boost the additive energy, their contribution to the distribution of pair correlations is asymptotically negligible. To control the contribution of arithmetic blocks we will use tools from metric Diophantine approximation.\\

The key point of the construction lies in the fact that the arithmetic blocks which are used in the construction 
have different prime numbers as their step sizes.\footnote{We will need to use a ``recycling
process'' for the prime moduli, since there are not enough different primes of the appropriate size available to have a different step size for each arithmetic block.} The fact that the step sizes of the arithmetic blocks are prime numbers will play a key role in two parts of the proof. On the one hand, using the theory of continued fractions at some point we will be led to counting the number of solutions of a certain equation; the assumption that the step size is a (large) prime will imply that we only have to count solutions which are a multiple of that prime, thus effectively reducing the number of solutions. This will allow us to control the contribution which comes from elements contained within one and the same arithmetic block. On the other hand, to control the contribution of the interaction of elements from two different arithmetic blocks, we will use a variance estimate which boils down to counting the maximal number of solutions of a simple Diophantine equation. Again, the fact that the moduli are (different) primes will reduce the maximal number of solutions of the equation. We will add some further comments on the heuristic reasoning behind the proof after first formulating precisely the way in which our sequence is constructed.\\

\paragraph*{Notation.} We fix some $\varepsilon>0$. Throughout the rest of this paper 
we assume w.l.o.g.\ that $\varepsilon$ is ``small'', say $\varepsilon  <1/100$. 
We will use Landau notation $o,\,\mathcal{O}$, and Vinogradov symbols $\ll,\gg$, 
with their usual meaning in analytic number theory 
(i.e. $f\ll g$ meaning that $|f|$ is bounded
by a constant times $|g|$, for all possible arguments). 
The symbol $f \asymp g$ means that $f \ll g$ as well as $f \gg g$. 
If the implied constant depends on some parameter, we will
indicate the dependence by a corresponding subscript. 
However, we will not indicate any dependence on $\varepsilon$, 
since throughout the proof $\varepsilon$ is assumed to be fixed. 
We write $\lambda $ for the Lebesgue measure on $\mathbb{R}$.
Finally, we write $\lfloor \cdot \rfloor$ for the integer part of a real number.\\

In Lemma \ref{lemma_1} below we construct the moduli of the arithmetic blocks.

\begin{lem} \label{lemma_1}
\label{lem: defining the moduli}
There exist an index $j_0 \geq 1$ and a sequence $(m_j)_{j \geq j_0}$ of primes such that 
\begin{equation}
m_{j} \asymp j^{\nicefrac{1}{4}}, \label{eq: regime of the moduli quantified by their level}
\end{equation}
and such that $m_{j} \neq m_{i}$ whenever $j-3\log j< i < j$, for all $i,j \geq j_0$.
\end{lem}

Before proving the lemma, we briefly explain its meaning. 
The numbers $m_{j}$ will be the moduli of the arithmetic 
progressions in our construction. The first condition 
in the lemma says that these moduli are of asymptotic 
order roughly $j^{\nicefrac{1}{4}}$. The second condition 
guarantees that the step size of the $j$-th arithmetic 
progression is different from the step size of the 
$i$-th arithmetic progression, whenever $i$ is ``close'' to $j$. 
So arithmetic blocks whose indices are close by can never have the same step size, which will guarantee that there is no undesired interaction between such blocks (this will play a crucial role in the proof of Lemma \ref{lem_AA} below). On the other hand, if $i$ and $j$ are not close to each other, then the corresponding arithmetic blocks are allowed to have the same step size --- this is the ``recycling process'', which was mentioned in the preceding footnote, and which is necessary since the step sizes of the blocks grow more slowly than the indices of the blocks themselves. However, this will not cause any problems since the block sizes increase very quickly and any interaction between a block and some other block of much smaller cardinality will always be negligible.

\begin{proof}[Proof of Lemma \ref{lemma_1}]
To define the value of $m_j$ for all indices $j$ in the range $16^{d} \leq j < 16^{d+1}$,
$d \geq 0$, we note that the number of primes in the range 
\begin{equation} \label{prime_int}
\left(2^{d}, 2^{d+1}\right)
\end{equation}
is certainly at least $d^2$ for all sufficiently large $d$, by the prime number theorem. So assume that $d$ is ``large'', and let $p_{d,1}<\ldots<p_{d,d^2}$ denote the first $d^2$ primes 
in the interval \eqref{prime_int}. We set
\begin{equation} \label{constr_p}
m_{j}\coloneqq p_{d,r\left(j\right)}, \qquad 16^{d} \leq j < 16^{d+1},
\end{equation}
where $r\left(j\right)$ is the unique remainder when reducing $j$ mod $d^2$. Then (\ref{eq: regime of the moduli quantified by their level}) holds since $2^{d} = (16^d)^{1/4} \asymp j^{\frac{1}{4}}$. Furthermore, it can easily be seen that the second assertion of the lemma holds as well for sufficiently large $d$, since \eqref{constr_p} guarantees that $m_j$ cannot equal $m_i$ whenever $|i-j|$ is small. Observe here that $d^2$ is of order roughly $(\log j)^2$, and thus much larger than $3 \log j$ for all sufficiently large $j$ and $d$.
\end{proof}

Let $j_{0}$ be as in Lemma \ref{lem: defining the moduli}.
%
For $j \geq j_0$ we recursively define sets $P_{G}\left(j\right)$ and $P_{A}(j)$ by setting
\begin{eqnarray}
\label{eq: definition of geometric block}
P_{G}\left(j\right) & \coloneqq& \Bigl\{ 
2^{(\max P_{A}(j-1))}
+3^{j^{h}} : h=0,\ldots,2^{j}-1\Bigr\}, \label{def:geo} \\
P_{A}\left(j\right) & \coloneqq& \Bigl
\{ 2^{(\max P_{G}(j))}+m_j h:  h=0,\ldots,\bigl
\lfloor2^{j}/j^{\deelta}\bigr\rfloor\Bigr\}. \label{def:ari}
\end{eqnarray}
To make the construction well-defined we need to specify the initial value $\max P_{A}(j_0-1)$, which is necessary for \eqref{eq: definition of geometric block} in the case $j=j_0$; it does not matter what we choose, but let us agree that this quantity should be read as 1, and that accordingly $P_{G}\left(j_0\right)\coloneqq\bigl\{2
+3^{j_0^{h}}: h=0,\ldots,2^{j_0}-1\bigr\}$.\\

The set $P_G(j)$ is a (shifted) geometric progression for each $j$, while the set $P_A(j)$ is a (shifted) arithmetic progression for each $j$. The sets $P_G(j)$ and $P_A(j)$ 
are arranged in increasing order; more precisely, we have
\begin{equation} \label{inequ}
P_G(j) < P_A(j) < P_G(j+1) 
\end{equation}
for all $j\geq j_0$, where the symbol ``$<$'' means that every element of the set on the right side exceeds every element of the set on the left side.\\

The exponential factors $2^{(\dots)}$ which appear in the definitions of all the sets $P_G$ and $P_A$ are quite arbitrary; what matters is only that the smallest element of $P_A(j)$ is much larger than the largest elements of $P_G(j)$, and so on. Thus the respective sets in our construction are not only ordered as shown by \eqref{inequ}, but there actually are huge gaps separating one item in this chain of inequalities from the next.\\ 

Finally, we specify the sequence $(a_n)_{n}$ by defining $a_{n}=a_{n}\left(\varepsilon\right)$ 
as the $n$-th (smallest) element of
\[
\bigcup_{j\geq j_0} \bigl(P_{G}\left(j\right)\cup P_{A}\left(j\right)\bigr),
\]
for all $n \geq 1$. So $(a_n)_{n}$ contains all the numbers which are contained in $P_G(j)$ or $P_A(j)$ for some $j$, sorted in increasing order. We claim that the additive energy of this sequence is as large as specified in \eqref{eq: bound for the additive energy of truncations of the final sequences}, and that the pair correlations of $(a_n \alpha)_{n}$ mod 1 are Poissonian for almost all $\alpha$.\\

\subsection{The heuristic behind the construction}

Before turning to the proof of Theorem \ref{thm: main theorem}, we want to explain the heuristic behind the construction of the sequence $(a_n)_{n}$. In particular, we want to show why our construction allows to go beyond the alleged ``Khintchine threshold''. Note that the distribution of the pair correlations of $(a_n \alpha)_{n}$ mod 1 depends not so much on the sequence $(a_n)_{n}$ itself, but rather on the set of differences $\{a_n - a_m\}_{m,n}$ (as does the additive energy). Thus it is this difference set that has to be controlled.\\

Obviously the difference set of a finite arithmetic progression has a very special structure; it is essentially an arithmetic progression itself, and the cardinality of the difference set of an arithmetic progression is small while the additive energy is large. More precisely, the positive part of the difference set of an arithmetic progression with step size $d$ and length $M$ is itself an arithmetic progression with step size $d$, and length $M-1$, and each of the elements of the difference sets has at least $1$ and at most $M-1$ representations as a difference of elements of the original set. In our construction we combine arithmetic progressions with different prime moduli $m_{j}$. The number of such arithmetic progressions and their respective length is so large that they boost the additive energy of the total sequence; in contrast, we have to show that their contribution to the distribution of pair correlations is asymptotically negligible. In our setting, at the $j$-th building block we have constructed roughly $N \approx 2^j$ elements of our sequence $(a_n)_{n}$. Each arithmetic progression at this level consists of roughly $\approx N / (\log N)^{1/4+\varepsilon/3}$ elements. One can easily check that this leads to the required lower bound for the additive energy. The size of the prime moduli $m_j$ is roughly $(\log N)^{1/4}$.\\

To make sure that the contribution which one of these arithmetic progressions makes to the pair correlations is asymptotically negligible, we have to show (roughly speaking) that
\begin{equation} \label{exc}
\frac{N}{(\log N)^{\deelta}}  \cdot \# \Big\{ q \leq N / (\log N)^{\deelta}: ~\|m_{j} q \alpha\| \leq \frac{1}{N} \Big\} = o(N)
\end{equation}
for ``typical'' $\alpha$ in the sense of Lebesgue measure (where for simplicity we took $1/N$ rather than $s/N$ for the length of the test interval). Here the factor $N/(\log N)^{\deelta}$ on the very left arises as the maximal number of representations which the number $m_{j} q$ has as a difference of two elements of the arithmetic progression with step size $m_{j}$, and the upper bound $q \leq N / (\log N)^{\deelta}$ which restricts the maximal size of $q$ comes from the length of the arithmetic blocks. The estimate \eqref{exc} is true as long as the cardinality of the set on the left is asymptotically negligible in comparison to  $(\log N)^{\deelta}$.\\

Essentially, the cardinality of this set only exceeds $(\log N)^{\deelta}$ when there is a 
$$q \in \{1, \dots, N/(\log N)^{\nicefrac{1}{2} 
+ \nicefrac{2 \varepsilon}{3}}\}$$ 
such that $\|m_{j} q \alpha\|$ is less than 
$1/(N  (\log N)^{\deelta})$, so that for the next 
$(\log N)^{\deelta}$ multiples of $q$ we also have 
$\|m_{j} q \alpha\| \leq 1/N$, and so that all 
these multiples are still smaller than 
$N/(\log N)^{1/4+\varepsilon}$. 
Accordingly, one has to check if for typical $\alpha$ 
one should expect that there is a $q$ such that
\begin{equation*} 
\|m_{j} q \alpha\| \leq \frac{1}{N (\log N)^{\deelta}}, \qquad 1 \leq q \leq N/(\log N)^{\nicefrac{1}{2} + \nicefrac{2 \varepsilon}{3}}.
\end{equation*}
Writing $N/(\log N)^{\nicefrac{1}{2} 
+ \nicefrac{2 \varepsilon}{3}}=:Q$, the inequality above essentially becomes
\begin{equation} \label{Q_ri}
\|m_{j} q \alpha\| \leq \frac{1}{Q (\log Q)^{3/4 + \varepsilon}}, \qquad 1 \leq q \leq Q.
\end{equation}
By Khintchine's convergence/divergence criterion this inequality looks like it should have infinitely many solutions for ``typical'' alpha, since the expression on the right-hand side is not summable as a function of $Q$. However, one major aspect is missing. The right-hand side of \eqref{Q_ri} is so small that the solutions of this inequality can be explicitly characterised by continued fraction theory; all solutions necessarily come from best approximations to $\alpha$. We will show that (for typical alpha) we may assume that the denominators of best approximations to alpha are not divisible by the prime $m_j$; thus the number $m_j q$ cannot be a best approximation denominator itself. Rather, it must be the multiple of $m_j$ and of a best approximation denominator, and accordingly for $q$ itself to satisfy \eqref{Q_ri} we must have
\begin{equation} \label{Q_ri_2}
\|q \alpha\|= \frac{\|m_j q \alpha\|}{m_j} \leq \frac{1}{m_j Q (\log Q)^{3/4 + \varepsilon}} \approx \frac{1}{Q (\log Q)^{1+\varepsilon}}, \qquad 1 \leq q \leq Q.
\end{equation}
The right-hand side of \eqref{Q_ri_2} is summable as a function of $Q$, and thus by Khintchine's criterion we should only expect finitely many solutions for typical alpha. It turns out that this heuristic reasoning can be turned into an actual proof.\\

We emphasise again that the fact that $m_j$ always is a prime played a crucial role in this reasoning, together with the fact that we may assume that the denominators of best approximations are not divisible by $m_j$ (we will prove this fact in Lemma \ref{lem:div} below). Another crucial aspect is to show that two different arithmetic blocks do not ``interact'' in an undesired way; that is, we have to show that the difference sets of these respective progressions do not overlap too much. For this it will again be important that all the moduli are (different) primes, since then a fixed integer can only show up in the difference set of two arithmetic progressions if it is a product of the two primes which constitute the respective step sizes. This will be proved in the form of a variance bound in Lemma \ref{lem_AA}.\\
%

Finally, let us remark why it is not possible to obtain 
even larger additive energy with such a construction. 
Obviously, the additive energy is increased when 
the length of the arithmetic blocks is increased, 
so we might try to do that. Furthermore, as \eqref{Q_ri_2} shows, 
increasing the size of the prime moduli $m_j$ 
would also improve the argument, 
so we might try to do that as well. 
So let us assume that the length 
of the arithmetic blocks is changed from roughly 
$N/(\log N)^{\deelta}$ to $N/(\log N)^\beta$ 
for some $\beta$, and that the size 
of the prime moduli $m_j$ is changed from roughly 
$(\log N)^{1/4}$ to $(\log N)^\gamma$ 
for some $\gamma$. If we do so, then instead of 
\eqref{exc} we will have to show that
\begin{equation} \label{exc2}
\frac{N}{(\log N)^{\beta}} \# \Big\{ q \leq N / (\log N)^{\beta}: ~\|m_{j} q \alpha\| \leq \frac{1}{N}  \Big\} = o(N)
\end{equation}
for ``typical'' alpha, with $m_j$ of size roughly $m_j \approx (\log N)^\gamma$. Now recall that Legendre's theorem from continued fraction theory allows us to characterise the solutions $(a,b)$ to $|b \alpha - a | < 1/(2b)$. We want to use this for $b = m_j q$, and thus in our application $b$ might be as large as $N (\log N)^{\gamma-\beta}$. The term $1/N$ in \eqref{exc2} is preassigned, since it comes from the definition of pair correlations. So in order to apply Legendre's theorem we have to make sure that $N (\log N)^{\gamma-\beta} \ll N$, which implies $\gamma \leq \beta$. This restricts the size of the prime moduli (in terms of the length of the arithmetic progressions). When we carry out the heuristic reasoning above with general parameters $(\beta,\gamma)$ instead of $(1/4+\varepsilon/3,1/4)$, then instead of \eqref{Q_ri_2} we will arrive at
\begin{equation}
\|q \alpha\| \leq \frac{1}{Q (\log Q)^{3 \beta+\gamma}}, \qquad 1 \leq q \leq Q.
\end{equation}
The right-hand side is summable if $3 \beta + \gamma >1$. Since the additive energy is maximised by taking $\beta$ as small as possible, and since we already know that we need to take $\gamma \leq \beta$, the minimal permissible value for $\beta$ is restricted by the requirement $\beta > 1/4$. This is the choice of parameters which is made in our construction. One can also show that our choice of parameters is optimal with respect to the conditions imposed by the variance bound in Lemma \ref{lem_AA}, which also requires that $3 \beta + \gamma >1$. Thus some significant new ideas would be necessary to further increase the additive energy while preserving the metric pair correlation property.

\subsection{A useful partition, and organisation of the paper} \label{split}
The following partitioning underpins the remaining part of this paper.
For doing so, we need some notation from additive combinatorics:
We write $X-Y$ for the difference set 
\[
X-Y\coloneqq\left\{ x-y:\,x\in X,\,y\in Y\right\} 
\]
of two sets $X,Y\subseteq\mathbb{Z}$. By $\#X$ we denote the cardinality
of $X$. Furthermore, we write $\mathrm{r}_{X-Y}$ 
for the number of ways in which $d\in\mathbb{Z}$ can be represented as a
difference of elements of $X,Y\subseteq\mathbb{Z}$, that is,
\begin{equation}\label{eq: definition of rep func}
\mathrm{r}_{X-Y}\left(d\right)
\coloneqq\#\left\{ \left(x,y\right)\in X\times Y:\,d=x-y\right\} .
\end{equation}
If no confusion can arise, we will simply write $\mathrm{r}\left(d\right)$
for $\mathrm{r}_{X-Y}\left(d\right)$. Recall that trivially 
\begin{equation*}\label{eq: trivial estimate for rep func}
\mathrm{r}_{X-Y}\left(d\right)\leq\min\left\{ \#X,\#Y\right\}.
\end{equation*}
Moreover, let $X^{+}\coloneqq X\cap\mathbb{Z}_{\geq1}$ denote the
set of positive elements of a set $X\subseteq\mathbb{Z}$. Since $A_{N}-A_{N}$
is symmetric around the origin, we can confine attention to its positive
part.\\

Assume that $d \geq 0$ is the difference of two elements of $A_N$, that is, $d=x-y$. We will classify these differences, according to the origin of $x$ and $y$. More precisely, we will distinguish between the following cases.
\begin{itemize}
\item Case (GG): $x$ and $y$ both come from geometric blocks, that is, $x,y \in \bigcup_j P_G(j)$.
\item Case (AG): $x$ comes from an arithmetic, and $y$ comes from a geometric block, that is, $x \in \bigcup_j P_A(j)$ and $y \in \bigcup_{j} P_G(j)$. Or, reciprocally, $x$ comes from a geometric block and $y$ comes from an arithmetic block.
\item Case (AA\textsubscript{diff}): $x$ and $y$ come from \emph{different} arithmetic blocks, that is, $x \in P_A(j_1)$ for some $j_1$ and $y \in P_A(j_2)$ for some $j_2$, such that $j_1 \neq j_2$.
\item Case (AA\textsubscript{same}): $x$ and $y$ come from \emph{the same} arithmetic block, that is, $x,y \in P_A(j)$ for some $j$.
\end{itemize}

We write $\mathcal{D}_N \left(GG \right)$ for the set of those $d$ in the difference set $(A_N-A_N)^+$ which can be represented as Case (GG). In a similar way, we define $\mathcal{D}_{N} \left(AG\right),\mathcal{D}_{N}\left(AA\textsubscript{diff}\right)$, and $\mathcal{D}_{N}\left(AA_\textsubscript{same}\right)$.\\

The function $R$ which was defined in (\ref{eq: definition pair correlation function}) can be decomposed in a similar way in the form
\begin{equation} \label{decom:R}
R = R\left(GG\right)+R\left(AG\right) + R\left(AA\textsubscript{diff}\right)+R\left(AA\textsubscript{diff}\right).
\end{equation}
For this decomposition, we set
\begin{equation}\label{eq: definition of R upper}
R\left(GG\right)\coloneqq R
\left(GG,\alpha,s,N\right)  \coloneqq\frac{2}{N}
\sum_{d\in \mathcal{D}_{N}\left(GG \right)}
\mathrm{r}\left(d\right)I_{s,N}\left(d\alpha\right), \quad 
I_{s,N}\left(x\right) \coloneqq
\begin{cases}
1 & \left\Vert x\right\Vert \leq s/N\\
0 & \mathrm{otherwise,}
\end{cases}
\end{equation}
where $r(d)$ counts 
only the number of Case (GG) representations which 
$d \geq 1$ has in the form $d=x-y$ such that $x,y \in A_N$. 
Note that the factor 2 in \eqref{eq: definition of R upper}, 
which is not present in \eqref{eq: definition pair correlation function}, 
comes from the fact that we restricted ourselves 
to the positive part of the difference set $A_N-A_N$. 
Similarly, we define $R\left(AG\right),
~R\left(AA\textsubscript{diff}\right)$ 
and $R\left(AA\textsubscript{same}\right)$, 
where the function $r(d)$ is instead restricted 
to representations of $d$ as Case (AG), 
Case (AA\textsubscript{diff}), 
and Case (AA\textsubscript{same}), respectively.\\

By using the same methods as in \cite{all}, one can easily conclude that 
\begin{equation}
R\left(GG,\alpha,s,N\right)
\rightarrow2s\label{eq: counting of the geometric parts is Poissonian}
\end{equation}
as $N\rightarrow\infty$, for almost all $\alpha\in\left[0,1\right]$
and each $s>0$. This follows from the fact that geometric progressions 
have small additive energy, and the fact that the cardinality 
of the geometric blocks is dominant over the total cardinality 
of the arithmetic blocks which implies that $1/N$ really is the 
correct normalisation factor such that 
$R\left(GG\right)$ converges as desired for $N \to \infty$.\\

Thus it remains to show that all the remaining terms 
$R\left(AG\right),R\left(AA\textsubscript{diff}\right)$ 
and $R\left(AA\textsubscript{same}\right)$
vanish in the limit $N \to \infty$, for almost all $\alpha$.\\

The outline of the next sections is as follows. First, in Section \ref{sec: small differences}, we analyse the contribution of $R\left(AA\textsubscript{same}\right)$. Here Diophantine approximation 
determines the counting.\footnote{The mechanism furnishing these estimates 
is of a somewhat combinatorial nature,
and related to so-called {\em{Bohr sets}}. The combinatorial nature of these sets also plays a key role in a recent paper of Chow, cf.\ \cite{chow}.} Then, in Section \ref{sec: large differences}, we prove variance estimates to control $R\left(AG\right)$ and $R\left(AA\textsubscript{diff}\right)$.
Once these steps are completed, in Section \ref{sec: proof of main thm} we use
the Borel--Cantelli lemma with a 
sandwiching argument to finish the proof of Theorem \ref{thm: main theorem}.
\section{Analysing the contribution of the small differences}\label{sec: small differences}

Before proceeding further, we need to recall some notions and results
about continued fractions. For a (possibly finite) sequence 
$\left(\alpha_{i}\right)_{i}$
of strictly positive integers, we denote by 
\[
\alpha\coloneqq\left[\alpha_{1},\alpha_{2},\ldots\right]
=\frac{1}{\alpha_{1}+\frac{1}{\alpha_{2}+\frac{1}{\ddots}}}
\]
the associated (possibly finite) continued fraction in the unit interval
$\left[0,1\right]$. Moreover, let $p_{n}/q_{n}$ denote
the $n$-th convergent to $\alpha$. Then, the following are well-known
facts, cf.\ for instance \cite[Ch.1]{bugeaud}.
\begin{enumerate}
\item Legendre's theorem: If $\nicefrac{a}{b}$ 
is a fraction with 
$$
\left|\alpha-\frac{a}{b}\right| < \frac{1}{2b^{2}},
$$
then $\nicefrac{a}{b}$ is a convergent to $\alpha$.
\item We have
\begin{equation}
\left|\alpha-\frac{p_{n}}{q_{n}}\right|\asymp\frac{1}{\alpha_{n}q_{n}^{2}}, \label{eq: classic approximation estimates for quality of coninued fractions}
\end{equation}
where the implied constants are independent of $\alpha$.
\item Borel--Bernstein theorem: Let $B\coloneqq\left(b_{n}\right)_{n}$ be
a sequence of (strictly) positive real numbers, and consider the series
\begin{equation}
\sum_{n\geq1}\frac{1}{b_{n}}.\label{eq: series involving the comparison of growth in Borel-Bernstein theorem}
\end{equation}
If $V_{B} \subset [0,1]$ denotes the set of 
numbers $\alpha=\left[\alpha_{1},\alpha_{2},\ldots\right]$
satisfying $\alpha_{n}\leq b_{n}$ 
for all sufficiently large $n\geq1$, then 
\[
\lambda\left(V_{B}\right)=\biggl\{\begin{array}{l}
1\quad\mathrm{if}~\eqref{eq: series involving the comparison 
of growth in Borel-Bernstein theorem}\,\mathrm{~converges,}\\
0\quad\mathrm{if}~\eqref{eq: series involving the comparison 
of growth in Borel-Bernstein theorem}\,\mathrm{~diverges.}
\end{array}
\]
\end{enumerate}

\begin{lem} \label{lem:div}
Let $(m_j)_{j \geq 1}$ be the sequence 
of primes from Lemma \ref{lemma_1}, 
which was used in \eqref{def:ari} 
for the construction of our sequence. 
Then for almost all $\alpha \in [0,1]$ 
there exist only finitely many pairs of indices $(j,n)$ 
such that the prime $m_j$ divides $q_n$, 
and such that additionally $q_n / m_j \in [2^j/j^2,2^j]$, 
where $q_n$ is the denominator of a convergent to $\alpha$.
\end{lem}
\begin{proof}
Assume that the denominator $q_n$ of a convergent 
is divisible by a prime $m_j$, i.e. there is a $k$ such that $q_n = k m_j$. 
When $q_n$ is a convergent to $\alpha$ then 
$\|q_n \alpha\| \leq 1/q_n$, and thus 
$\| k m_j \alpha \| \leq 1/(k m_j)$. 
Thus to prove the lemma we have 
to show that almost all $\alpha \in [0,1]$ 
are contained in at most finitely many sets 
of the form
$$
S_{j,k} := \left\{ x \in [0,1]: \|k m_j x\| 
\leq \frac{1}{k m_j} \right\}, \qquad j=1,2,\dots, 
\quad 2^j/j^2 \leq k \leq 2^j. 
$$
We have
$$
\lambda(S_{j,k}) = \frac{2}{k m_j}.
$$
Furthermore, we have
$$
\sum_{j=1}^\infty ~\sum_{2^j/j^2 \leq k \leq 2^j} \frac{2}{k m_j} \ll \sum_{j=1}^\infty \frac{\log j}{m_j}.
$$
Recall that to construct our sequence $(m_j)_{j \geq 1}$ in Lemma \ref{lemma_1} we selected $d^2$ primes from the range $\left(2^{d}, 
2^{d+1}\right)$, for each (sufficiently large) $d$. Thus 
$$
\sum_{j=1}^\infty \frac{\log j}{m_j}  \ll \sum_{d} \frac{(\log \log d) d^2}{2^{d}} < \infty.
$$
Thus, by the Borel--Cantelli lemma,
almost all $\alpha$ are contained 
in only finitely many sets $S_{j,k}$.
\end{proof}

\begin{lem}
\label{prop: controlling the number of high energy multiples which are counted } Let
$$
M_j \coloneqq\bigl\{ q\leq2^{j}/j^{\deelta}:\,\bigl
\Vert m_j q\alpha\bigr\Vert\leq s/2^{j}\bigr\}.
$$
Then for almost all $\alpha \in [0,1]$ we have 
\[
\#M_j \ll_{s}\,j^{1/4}.
\]
\end{lem}
\begin{proof}
During this proof we suppress the potential dependence of the symbols ``$\ll$'' and ``$\gg$'' on $s$.\\

Let $B$ denote the sequence $(n^{1+\nicefrac{\varepsilon}{3}})_{n}$,
and suppose that $\alpha\in V_{B}$ is an irrational number (recall that the set $V_B$ was defined in the statement of the Borel--Bernstein theorem, before the statement of Lemma \ref{lem:div}). By the
Borel--Bernstein theorem, $V_B$ has full Lebesgue
measure. In the sequel we will assume that $\alpha \in [0,1]$ is a fixed number which is contained in $V_B$, and for which the conclusion of Lemma \ref{lem:div} holds. Note that the set of such $\alpha$'s has full Lebesgue measure.\\

Let us note the following. Let $q_m$ be the denominator of a convergent to $\alpha$. Assume that 
\begin{equation} \label{equ_fr}
\|q_m \alpha\| \leq \frac{s}{2^j}. 
\end{equation}
Then, as noted above, we have
$$
\|q_m \alpha\| \asymp \frac{1}{\alpha_m q_m}.
$$
Since $q_m$ grows at least exponentially in $m$, and since $\alpha \in V_B$ implies that $\alpha_m \ll m^{1+ \nicefrac{\varepsilon}{3}} \ll (\log q_m)^{1 + \nicefrac{\varepsilon}{3}}$, we thus see that \eqref{equ_fr} is only possible if
$$
\frac{1}{(\log q_m)^{1 + \nicefrac{\varepsilon}{3}} q_m} \ll \frac{1}{2^j},
$$
which in turn is only possible if 
$$
q_m \gg \frac{2^j}{j^{1+ \nicefrac{\varepsilon}{3}}}. 
$$

Now we argue in two steps.\\
\noindent (i) We first claim the following. If $j$ is
large enough, and if $M_j$ is non-empty, then there exists a unique value of $n$ such that $q_n$ is the denominator of a convergent to $\alpha$, such that $q_n \geq 2^j / j^{2}$, and such that
\begin{equation}
M_j \subseteq q_{n}\mathbb{Z}.\label{eq: rewriting the differences of arithmetic blocks in terms of cf}
\end{equation}
Indeed, if some $q$ is contained in $M_j$, then for this $q$ we have
$\left\Vert m_j q\alpha\right\Vert \leq s / 2^j < 1/(2 m_j q)$, if $j$ is sufficiently large. This is a consequence of our construction, where we have $m_j \asymp j^{1/4}$ and $q \leq 2^j/j^{\deelta}$. Let $p\in\mathbb{Z}$ be such that 
$\left\Vert m_j q\alpha\right\Vert =\left|m_j q\alpha-p\right|$.
Then Legendre's theorem implies that there is some $n\geq1$ with
\begin{equation}
\frac{p}{m_j q}=\frac{p_{n}}{q_{n}}. \label{eq: counted fractions are trivial transformations of cf}
\end{equation}
As a consequence, since $p_n$ and $q_n$ are coprime, 
there is some integer $g \geq 1$ such that $m_j q = g q_n$ and $p = g p_n$. 
Then we have $\|m_j q \alpha\| = |g q_n \alpha - g p_n| = g |q_n \alpha - p_n|$, 
and from the reasoning following equation \eqref{equ_fr} 
we can deduce that $q_n \gg 2^j/j^{1+\nicefrac{\varepsilon}{3}}$. 
Thus, provided that $j$ is sufficiently large, 
$q_n/m_j$ lies in the range $[2^j/j^2,2^j]$, and then, 
by Lemma \ref{lem:div}, 
we can assume that $q_n$ is not divisible by $m_j$.\\

Since we have now figured out that we may assume that $m_j$ does not divide $q_n$, we conclude that $p$ and $q$ can actually both be written in the form $p=h m_j p_{n}$
and $q=h m_j q_{n}$ for some integer $h \geq 1$. Observe that
(\ref{eq: classic approximation estimates for quality of coninued fractions})
implies
\begin{equation}
\frac{m_j h}{\alpha_{n}q_{n}}\asymp\bigl\Vert 
m_j q\alpha\bigr\Vert\leq\frac{s}{2^{j}},\label{eq: order of maginitude of approximation quality for multiples of convergents}
\end{equation}
and thus
$$
\alpha_n q_n \gg m_j h 2^j \gg j^{1/4} 2^j.
$$
Thus the well-known recursion $q_{n+1}=\alpha_{n}q_{n}+q_{n-1}$ 
yields $q_{n+1}\geq \alpha_{n}q_{n}\gg j^{1/4} 2^j$ 
for sufficiently large $j$. However, $M_j$ by definition 
is a subset of $\{1, \dots, 2^j/j^{\deelta}\}$. This shows that $q_{n+1}$ is already too large to be contained in $M_j$, and consequently $M_j$ consists only of integer multiples of $q_n$.

\medskip{}
\noindent (ii) Now we give an upper bound for the largest possible value of $h\geq1$ such that
$h m_j q_{n}\in M_j$. From 
(\ref{eq: order of maginitude of approximation 
quality for multiples of convergents}) and the definition of $M_j$ we deduce that
\[
h \leq\frac{2^{j}}{q_{n}j^{\deelta}}\qquad \text{as well as}
\qquad h \ll \frac{\alpha_{n}q_{n}}{m_j 2^{j}}.
\]
As noted above we have $2^j/j^{1+\nicefrac{\varepsilon}{3}} \ll q_n \ll 2^j/j^{\deelta}$. Thus
\[
h
\ll
\max_{\frac{2^{j}}{j^{1+\nicefrac{\varepsilon}{3}}}\ll x
\leq\frac{2^{j}}{j^{\deelta}}}
\min\left\{ \frac{2^{j}}{xj^{\deelta}},\,
\frac{\alpha_{n}x}{m_j 2^{j}}\right\} 
\]
where the $x\in\mathbb{R}$ maximising the right hand side, under
the given constraints, is determined via
\[
\frac{2^{j}}{xj^{\deelta}}=\frac{\alpha_{n}x}{m_j 2^{j}} \quad \Leftrightarrow \quad x^{2}=\frac{m_j 2^{2j}}{j^{\deelta}\alpha_{n}}.
\]
Thus, using $\alpha_n \ll (\log q_n)^{1+\nicefrac{\varepsilon}{3}} \ll j^{1+\nicefrac{\varepsilon}{3}}$, we finally obtain
\[
h^{2}\ll\frac{\alpha_{n}}{j^{\deelta}m_j} \ll j^{1/2}.
\]
Thus $\# M_j \ll j^{1/4}$, which proves the lemma.
\end{proof}

\section{Analysing the contribution of the large differences}\label{sec: large differences}

The Fourier series expansion of the indicator functions $I_{s,N}\left(\alpha\right)$ is given by
\begin{equation}
I_{s,N}\left(\alpha\right)\sim\sum_{n\in\mathbb{Z}}c_{n}e\left(n\alpha\right)\qquad\mathrm{where}\quad c_{n}\coloneqq\begin{cases}
\sin\left(2\pi ns/N\right)/\left(\pi n\right) & \mathrm{if\,}n\neq0,\\
2s/N & \mathrm{if\,}n=0,
\end{cases}\label{eq: Fourier expansion of the indicator functions}
\end{equation}
where we write $e\left(\alpha\right)$ for $\exp\left(2\pi i\alpha\right)$.
The next lemma is of a technical nature, and is used in a decoupling
argument for the variance bounds, 
which are derived in Section \ref{sec: Variance Bounds}.
\begin{lem} \label{lem:fourier}
Define for integers $u,v>0$ the quantity
\begin{equation}\label{eq: definition of Fourier weights C}
C\left(u,v\right)\coloneqq \sum_{\substack{n_{1},n_{2} \in\mathbb{Z} \setminus\left\{ 0\right\},\\n_{1}u=n_{2}v}} c_{n_{1}}c_{n_{2}}.
\end{equation}
Then
\begin{equation}
C\left(u,v\right)\ll\frac{\mathrm{gcd}\left(u,v\right)}
{\max\left\{ u,v\right\} }. \label{eq: auxiliary estimate of the Fourier coefficients in variance estimate}
\end{equation}
Moreover, for $u\neq0$ we have 
\begin{equation}
C\left(u,u\right)\ll_{s}N^{-1}.\label{eq: bound for the Fourier coefficients on the diagonal  in variance estiamte}
\end{equation}
\end{lem}
\begin{proof}
Note that $n_{1}u=n_{2}v$ holds if and only
if there is an integer $h\neq0$ satisfying 
$n_{1}=hu/\mathrm{gcd}\left(u,v\right)$
and $n_{2}=hv/\mathrm{gcd}\left(u,v\right)$. Moreover, we observe
that $\left|c_{n}\right|\leq\min\left\{ 2s/N,1/\left|n\right|\right\} $
for $n\neq0$. Combining these estimates with the Cauchy--Schwarz
inequality yields 
\begin{align*}
\left|C\left(u,v\right)\right|^{2} & \leq\sum_{h\in\mathbb{Z}
\setminus\left\{ 0\right\} }c_{h\frac{u}{\mathrm{gcd}\left(u,v\right)}}^{2}
\sum_{h\in\mathbb{Z}\setminus\left\{ 0\right\} }
c_{h\frac{v}{\mathrm{gcd}\left(u,v\right)}}^{2}\\
& \leq\sum_{h\in\mathbb{Z}
\setminus\left\{ 0\right\} }
\frac{\left(\mathrm{gcd}\left(u,v\right)\right)^{2}}
{\left(uh\right)^{2}}\sum_{h\in\mathbb{Z}
\setminus\left\{ 0\right\} }\frac{\left(\mathrm{gcd}
\left(u,v\right)\right)^{2}}{\left(vh\right)^{2}},
\end{align*}
which implies (\ref{eq: auxiliary estimate of the Fourier coefficients in variance estimate}).\\

Furthermore,
\[
C\left(u,u\right)
\ll\sum_{n\leq\frac{N}{2s}}
\frac{4s^{2}}{N^{2}}+\sum_{n>\frac{N}{2s}}\frac{1}{n^{2}},
\]
which implies (\ref{eq: bound for the Fourier coefficients 
on the diagonal  in variance estiamte}).
\end{proof}
\noindent From orthogonality relations,
combined with (\ref{eq: Fourier expansion of the indicator functions}),
we obtain
\begin{eqnarray}
N^{2} ~\mathrm{Var}\bigl(R\left(AG,\cdot,s,N\right)\bigr)
& =& \int_{0}^{1}\Biggl(\sum_{d\in \mathcal{D}_{N}\left(AG\right)}\mathrm{r}\left(d\right)
\sum_{n\in\mathbb{Z}\setminus\left\{ 0\right\} }c_{n}e
\left(dn\alpha\right)\Biggr)^{2}\mathrm{d}\alpha \nonumber\\
& = & \sum _{u,v\in \mathcal{D}_{N}\left(AG\right)}
~\mathrm{r}\left(u\right)\mathrm{r}
\left(v\right)C\left(u,v\right), \label{eq: variance of block counting function in terms of Fourier coefficients}
\end{eqnarray}
where $r(\cdot)$ is the representation function which counts representations as Case (AG). A perfect analogue holds when (AG) is replaced by (AA\textsubscript{diff}) everywhere in the formula (including in the definition of the representation function $r$).\\

The main term on the right hand side, as we shall see, is the sum
over the diagonal $\left(\mathrm{r}
\left(u\right)\right)^{2}C\left(u,u\right)$.
To prove this, the next lemma shows that the contribution
from the off-diagonal terms is small. 
More precisely, $C(u,v)$ is extremely small for two elements 
$u \neq v$ of $\mathcal{D}_{N}(AG)$ 
or $\mathcal{D}_{N}(AA\textsubscript{diff})$.

\begin{lem} \label{lemma:off}
We have
\begin{equation}
\sum_{\substack{u,v \in \mathcal{D}_{N}(AG),\\ u \neq v}}
\mathrm{r}\left(u\right)\mathrm{r}\left(v\right)C\left(u,v\right)
\ll 1,\label{eq: bound on the off-diagonal Fourier contribution in variance estimate}
\end{equation}
where the representation function $r$ counts representations from Case (AG). 
The same estimate holds if $(AG)$ is replaced by $(AA\textsubscript{diff})$.
\end{lem}
\begin{proof}
This is not a critical part in the whole argument, 
and it is sufficient to use very rough estimates. 
We only give a brief outline of the proof. 
Let $u$ and $v$ be elements of the difference set 
$\mathcal{D}_{N}(AG)$  such that $0 < u < v$. 
Recall that different building blocks of 
our sequence are separated by huge constants. 
For $u$ and $v$ this leaves only two possibilities: 
\begin{itemize}
\item Either $u$ is of much smaller order than $v$, say $u \ll v^{1/2}$. 
By \eqref{eq: auxiliary estimate of the Fourier coefficients 
in variance estimate} we have $C(u,v) \ll \gcd(u,v)/\max\{u,v\}$. 
Since $\gcd(u,v) \leq u$, we have $C(u,v) \ll u/v \ll v^{-1/2}$.
\item The second possibility is that $u$ and $v$ are of very similar size, 
and that consequently $v-u$ is very small in comparison with $v$. 
In this case we may assume for example that $v-u \ll v^{1/2}$. 
Again using $C(u,v) \ll \gcd(u,v)/\max\{u,v\}$, 
and now observing that $\gcd(u,v) \leq v-u \ll  v^{1/2}$, 
we obtain $C(u,v) \ll v^{-1/2}$. 
\end{itemize}
So in both cases $C(u,v)$ is small in comparison with $v$. 
By construction the difference set $\mathcal{D}_{N}(AG)$ 
is an extremely sparse set, due to the very fast growth of our sequence. 
Thus after summing over $u$ and $v$ 
we can obtain \eqref{eq: bound on the off-diagonal Fourier contribution in variance estimate}. A similar argument works 
when instead of $\mathcal{D}_{N}(AG)$ 
we consider $\mathcal{D}_{N}(AA\textsubscript{diff})$.
\end{proof}

\subsection{Variance bounds}\label{sec: Variance Bounds}
Now we have the tools at hand to derive the variance bounds 
for the auxiliary functions $R\left(AG,\cdot,s,N\right)$ and 
$R\left(AA\textsubscript{diff},\cdot,s,N\right)$ 
which were defined in (\ref{eq: definition of R upper}).\\

\begin{lem}
\label{prop: bounding the variance of the mixed differences of geom and arth part}For every fixed $s>0$,  we have
\begin{equation}
\mathrm{Var}\bigl(R\left(AG,\cdot,s,N\right)\bigr)
\ll_{s} N^{-1/2}.
\label{eq: variance estimate for the mixed differences geom and arithm}
\end{equation}
\end{lem}
\begin{proof}
Again this is not a crucial lemma, and it is sufficient to use very rough estimates. Note that trivially $\# \mathcal{D}_{N}\left(AG\right) \leq N^{2}$. Let $u\in \mathcal{D}_{N}\left(AG\right)$. Then, using again the fact that our sequence increases very quickly, we can easily show that the number of Case (AG) representations $r(u)$ which $u$ has as the difference of two elements from $A_N$ is very small. To give a quantitative statement, we could easily show that $r(u) \ll N^{1/4}$, uniformly in $u$ (this is just a very rough estimate). 
Hence (\ref{eq: bound for the Fourier coefficients on the diagonal  
in variance estiamte}) implies 
\[
\sum_{u\in \mathcal{D}_{N}\left(AG\right)} \mathrm{r}\left(u\right)^{2}
\left|C\left(u,u\right)\right| \ll_s \left(\# \mathcal{D}_{N}\left(AG\right)\right) N^{1/2} N^{-1} \ll_{s}  N^{3/2}.
\]
Together with (\ref{eq: variance of block counting function 
in terms of Fourier coefficients}) and 
(\ref{eq: bound on the off-diagonal Fourier 
contribution in variance estimate})
this implies (\ref{eq: variance estimate for the mixed 
differences geom and arithm}).\\
\end{proof}

The contribution coming from numbers which arise 
as the difference between two numbers from different arithmetic 
blocks is a bit more difficult to control. 
To see this, note that when there are two arithmetic progressions 
with different step sizes $m_{j_1}$ and $m_{j_2}$, 
then there are certain numbers which have many representations 
as a number from the first arithmetic progression, 
minus a number from the second arithmetic progression. 
To control the contribution from such numbers, 
we will make crucial use of the fact that 
in our construction the step sizes $m_{j_1}$ and $m_{j_2}$ 
are prime numbers.

\begin{lem} \label{lem_AA} For every fixed $s >0$, we have
\begin{equation}
\mathrm{Var}\bigl(R\left(AA\textsubscript{diff},\cdot,s,N\right)\bigr)
\ll_{s}
\frac{1}{\left(\log N\right)^{1+\varepsilon/2}}. \label{eq: variance estimates for the arithm and arithm diferences}
\end{equation}
\end{lem}

\begin{proof}
Let $N$ be given. There is some $J$ such that $a_N \in P_A(J) \cup P_G(J)$, and by construction for this value of $J$ we have $J \asymp \log N$. By \eqref{eq: variance of block counting function in terms of Fourier coefficients} and Lemma \ref{lemma:off} we have
\begin{eqnarray*}
\mathrm{Var}\bigl(R\left(AA\textsubscript{diff},\cdot,s,N\right)\bigr) & \leq & \frac{1}{N^2} \sum_{u,v \in \mathcal{D}_{N}\left(AA\textsubscript{diff}\right) } r(u) r(v) |C(u,v)| \\
& \ll & \frac{1}{N^2} \left( 1 +  \sum_{\substack{1 \leq j_1 < j_2 \leq J}} \quad 
\sum_{u\in P_{A}\left(j_2\right)-P_{A}\left(j_1\right)}
\mathrm{r}\left(u\right)^{2}\left|C\left(u,u\right) \right| \right),
\end{eqnarray*}
where $r(u)$ counts the number of representation of 
$u$ as the difference between an element of 
$P_{A}\left(j_2\right)$ and an element of 
$P_{A}\left(j_1\right)$. Here we used the fact 
that due to the huge constants which separate different blocks in our construction, for given $u$ 
there is only one pair $(j_1,j_2)$ such that 
$u \in P_{A}\left(j_2\right)-P_{A}\left(j_1 \right)$, 
except maybe for finitely many (small) values of $u$.\\

Let $j_1 < j_2$ be fixed. First assume that $j_1 < J-2\log J$. 
We note that the cardinality of the set
$P_{A}\left(j_2 \right)-P_{A}\left(j_1 \right)$ 
is bounded by
\begin{eqnarray}
\# \{P_{A}\left(j_2 \right)-P_{A}\left(j_1 \right) \} & \ll &
\max \{P_{A}\left(j_2 \right)-P_{A}(j_1) \} - 
\min \{P_{A}\left(j_2 \right)-P_{A}(j_1) \} \nonumber\\ 
& \ll & m_{j_2} \frac{N}{(\log N)^{\deelta}} \nonumber\\
& \ll & \frac{N}{(\log N)^{\nicefrac{\varepsilon}{3}}}. \label{papg}
\end{eqnarray}
Then by the trivial estimate 
$\mathrm{r}\left(u\right)\ll\#P_{A}\left(j_1\right) \ll N / 
(\log N)^{2 \log 2 +1/4}$,
and since $2 \log 2 +1/4 > 16/10$, we have
\begin{eqnarray}
\frac{1}{N^{2}} \sum_{\substack{j_1, j_2, \\ j_1 < J-2\log J}} 
\sum_{u\in P_{A}\left(j_2\right)-P_{A}\left(j_1\right)}
\mathrm{r}\left(u\right)^{2}\left|C\left(u,u\right)
\right| & \ll_s & \frac{J^2}{N^2} \frac{N^3}{(\log N)^{16/5}} \frac{1}{N} \nonumber\\
& \ll &
\frac{1}{\left(\log N\right)^{6/5}}. \label{eq: smallish part of variance of arithm differences}
\end{eqnarray}

It remains to control the contribution from the range 
$J-2\log J\leq j_1 < j_2 \leq J$. Here it plays a crucial role 
that for $j_1,j_2$ in this range, by construction 
there are two \emph{different} primes $m_{j_1}$ and $m_{j_2}$ 
which form the step sizes of the arithmetic progression $P_A(j_1)$ and $P_A(j_2)$, respectively (cf.\ Lemma \ref{lemma_1}). Therefore, in such a situation $\mathrm{r}\left(u\right)$ is bounded by the number of
solutions $\left(x,y\right)\in\mathbb{Z}^{2}$ to the linear Diophantine
equation
\[
\tilde{u}=m_{j_2}x-m_{j_1}y\qquad\mathrm{where}
\quad\tilde{u}\coloneqq u -
\min \{P_{A}(j_2)\}+\min \{P_{A}(j_1)\},
\]
and $(x,y)$ satisfies the additional restriction 
that $1\leq x,y\leq N/(\log N)^{\deelta}$. 
Since $m_{j_1}$ and $m_{j_2}$
are prime numbers, the set of integer solutions
to this equation admits the form 
\[
(x_{0}+hm_{j_1},y_{0}-h m_{j_2}),
\]
where $h\in\mathbb{Z}$ and $\left(x_{0},y_{0}\right)$ 
is some solution to the above equation.
Moreover, the size of $j_1$ and $j_2$, together with 
(\ref{eq: regime of the moduli quantified by their level}),
ensures that $m_{j_1} \asymp m_{j_2} \asymp (\log N)^{1/4}$.
Hence,
\begin{equation}
\mathrm{r}\left(u\right)\ll
\frac{N}{\left(\log N\right)^{\nicefrac{1}{2}+\nicefrac{\varepsilon}{3}}}. \label{eq: estimate on number of representations for late blocks}
\end{equation}
Thus using \eqref{eq: bound for the Fourier coefficients on the diagonal  in variance estiamte}, \eqref{papg} and  (\ref{eq: estimate on number of representations 
for late blocks}), and noting that $\log J \ll \log \log N \ll (\log N)^{\varepsilon/2}$,  we obtain
that
\begin{eqnarray*}
\frac{1}{N^{2}}
\sum_{\substack{j_1,j_2,\\J-2\log J\leq j_1 < j_2 \leq J}} \underset{u\in P_{A}\left(j_2\right)-P_{A}
\left(j_1 \right)}{\sum} \mathrm{r}\left(u\right)^{2}\left|C\left(u,u\right)
\right| &\ll_s &
\frac{1}{N^{2}}
\sum_{\substack{j_1,j_2,\\J-2\log J\leq j_1 < j_2 \leq J}} \frac{N^3}{(\log N)^{1 + \varepsilon}} \frac{1}{N}  \\
& \ll & \frac{1}{\left(\log N\right)^{1+\nicefrac{\varepsilon}{2}}}.
\end{eqnarray*}
Combining this with (\ref{eq: smallish part of variance 
of arithm differences})
yields (\ref{eq: variance estimates for the arithm and arithm diferences}).
\end{proof}

\section{Proof of Theorem \ref{thm: main theorem}}\label{sec: proof of main thm}

Let $N$ be given. There is a number $J$ such that $a_N \in P_A(J) \cup P_G(J)$, and for this value of $J$ we have $J \asymp \log N$ and $2^J \asymp N$. Then
\[
E\bigl(A_{N}\bigr)\geq E
\bigl(P_{A}\left(J-1 \right)\bigr)
\gg \frac{N^{3}}{(\log N)^{3/4 + \varepsilon}},
\]
where we used that the additive energy 
of an arithmetic progression is proportional 
to the third power of its cardinality, 
and that by construction 
$\# P_{A}\left(J-1\right) \gg 2^{J}/J^{\deelta} \gg N/(\log N)^{\deelta}$. Thus the additive energy of the sequence constructed in our example is indeed as large as claimed in the statement of the theorem.\\

It remains to show that $(a_n)_{n}$ has the metric pair correlation property. Recall that the contribution coming from the geometric blocks gives the desired convergence $R\left(GG,\alpha,s,N\right)
\rightarrow 2s$ for almost all $\alpha$, for 
every fixed $s>0$, cf.\ (\ref{eq: counting of the geometric 
parts is Poissonian}). It is a standard procedure to use the variance estimates 
and the results from the previous section to conclude 
that the contribution of the parts $R\left(AG\right)$ and $R\left(AA\textsubscript{diff}\right)$
tends to zero in the limit; 
thus we will only give a brief outline. 
Fix a rational $s>0$. Define the sequence 
$$
N_{m}=\bigl\lfloor\exp\bigl(m^{\frac{1}{1+\varepsilon/2}}\bigr)\bigr\rfloor
$$
and note that $N_{m+1}/N_{m}\rightarrow1$. 
If $N$ is such that $N_{m}\leq N<N_{m+1}$, then 
\[
N R\bigl(AG,\alpha,s,N\bigr)
\leq N_{m+1} R
\bigl(AG,\alpha,N_{m+1}/N_{m}s,N_{m+1}\bigr).
\]
Denote by $E_{AG,s}\left(N_{m}\right)$ the ``exceptional'' set 
$$
\Big\{ \alpha\in\left[0,1\right]:\left| R
\bigl(AG,\alpha,N_{m}/N_{m+1}s,N\bigr)
-\mu_{AG,s}\left(N_{m}\right)\right|\geq1/\log\log N_{m}\Big\}
$$
where $\mu_{AG,s}\left(N_{m}\right)$ is the expected value of 
$R\bigl(AG,\alpha,N_{m}/N_{m+1}s,N\bigr)$. Observe that $\mu_{AG,s}\left(N_{m}\right)\rightarrow 0$ as $m\rightarrow\infty$,
since the indices of those elements of $(a_n)_n$ which
come from an arithmetic block are contained in a 
set of zero density within the total index set.
Combining Chebyshev's inequality with the variance estimates 
from Lemma
\ref{prop: bounding the variance of the mixed differences of 
geom and arth part} and Lemma \ref{lem_AA}, and applying the Borel-Cantelli lemma, we obtain 
\begin{equation}
R\bigl(AG,\alpha,s,N\bigr)
\underset{N\rightarrow\infty}{\longrightarrow}0,\label{eq: counting of interactions between different blocks generically goes to zero}
\end{equation}
for all rational $s$ and for Lebesgue almost all $\alpha\in\left[0,1\right]$. Exactly the same argument works if (AG) is replaced by (AA\textsubscript{diff}).\\

Finally we have to show that $R\left(AA\textsubscript{same}\right) \rightarrow 0$ for almost all $\alpha$. Let $s >0$ be fixed, and assume that $s$ is rational. By the Borel--Bernstein theorem, almost no $\alpha \in [0,1]$
has infinitely many $d\geq N/(\log N)^{3/2}$ such that $I_{s,N}(d\alpha)=1$. 
Hence it is sufficient to estimate 
the contribution of those differences $d$ which are contained in $(P_A(j) - P_A(j))^+$ for a value of $j$ which is close to $J$. More precisely,
we can restrict $j$ to the range $J-2\log J\leq j \leq J$. 
By Lemma \ref{prop: controlling the number of 
high energy multiples which are counted }, for almost all $\alpha\in[0,1]$ we have
\begin{eqnarray}
R\bigl(AA\textsubscript{same},\alpha,s,N\bigr)
& \ll & 
\frac{1}{N} \sum_{J-2\log J\leq j \leq J} \frac{2^j}{j^{\deelta}} \cdot \# \left\{d\in (P_A(j) - P_A(j))^+ :~\|d \alpha\| \leq \frac{s}{N} \right\} \nonumber \\
& \ll_{s} & \frac{1}{N} \sum_{J-2\log J\leq j \leq J} \frac{2^j}{j^{\nicefrac{\varepsilon}{3}}} \nonumber \\
& \ll & (\log N)^{-\nicefrac{\varepsilon}{6}}, \label{eq: counting of the high engery differences generically goes to zero}
\end{eqnarray}
where we estimated $\log J \ll \log \log N \ll (\log N)^{\nicefrac{\varepsilon}{6}}$.\\

Thus we have $R(GG) \to 2s$, and $R(AG) \to 0, ~R(AA\textsubscript{diff}) \to 0, ~R(AA\textsubscript{same})  \to 0$, for all rational $s>0$, for almost all $\alpha$. However, if this convergence holds for all rational $s>0$ and almost all $\alpha$, then by monotonicity it must also hold for all real $s>0$ and almost all $\alpha$. In view of the decomposition \eqref{decom:R} this concludes the proof of the theorem.

\paragraph*{Acknowledgements}
The authors would like to thank the anonymous referee 
for many valuable suggestions which significantly 
improved the presentation of this paper.
CA is supported by the Austrian Science Fund (FWF), 
projects Y-901 and F 5512-N26. TL is also supported by FWF project Y-901. 
NT is supported by FWF project W1230. The present work was,
to a non-trivial part, carried out while 
NT was visiting the number theory group of the University
of York. He wishes to thank this group for its warm hospitality, and
the pleasant memories. 


\end{document}